\documentclass[11pt]{amsart}
\usepackage{amsmath,amssymb,amsfonts,amsthm,amsopn}
\usepackage{latexsym,graphicx}
\usepackage{pb-diagram}

 \def\Spnr{Sp(d,\R)}
 \def\Gltwonr{GL(2d,\R)}

\newcommand{\stft}{short-time Fourier transform}

\newcommand{\modsp}{modulation space}

\newtheorem{theorem}{Theorem}[section]
\newtheorem{lemma}[theorem]{Lemma}

\newtheorem{proposition}[theorem]{Proposition}
\newtheorem{definition}[theorem]{Definition}

\newtheorem{example}[theorem]{Example}
\newtheorem{remark}[theorem]{Remark}

\newcommand{\beqa}{\begin{eqnarray*}}
\newcommand{\eeqa}{\end{eqnarray*}}

\newcommand{\field}[1]{\mathbb{#1}}
\newcommand{\bR}{\field{R}}        
\newcommand{\bN}{\field{N}}        
\newcommand{\bZ}{\field{Z}}        
\newcommand{\bC}{\field{C}}        
\def\G{\mathcal{G}}

\def\la{\lambda}

\def\eps{\epsilon}

\def\cS{\mathcal{S}}

\def\cG{\mathcal{G}}

\def\cA{\mathcal{A}}

\def\a{\aleph}

\def\rd{\bR^d}

\def\rdd{{\bR^{2d}}}
\def\zdd{{\bZ^{2d}}}

\def\lrd{L^2(\rd)}

\def\intrd{\int_{\rd}}

\def\R{\right)}

\def\<{\left<}
\def\>{\right>}

\def\inv{^{-1}}

\def\mv1{M_v^1}

\def\phas{(x,\o )}
\def\mn{(m,n)}
\def\mn'{(v_1,u_2')}

\def\Spnr{Sp(d,\R)}

\newcommand{\abs}[1]{\lvert#1\rvert}

\def\o{\eta}
\def\a{\alpha}
\def\b{\beta}

\def\R{\mathbb{R}}
\def\Ren{\mathbb{R}^d}

\def\sch{\mathcal{S}}

\def\Fur{\mathcal{F}}

\def\Sn2{S_{2}(L^{2}(\Ren))}
\def\S1{S_{1}(L^{2}(\Ren))}
\def\sig00{\sigma_{0,0}}

\def\la{\langle}
\def\ra{\rangle}

\newcommand{\A}{\mathcal{A}}

\begin{document}
\begin{abstract}
We investigate the sparsity of the Gabor-matrix representation of Fourier integral operators with a phase having quadratic growth. It is known that such an infinite matrix is sparse and well organized, being in fact concentrated along the graph of the corresponding canonical transformation. Here we show that, if the phase and symbol have a regularity of Gevrey type of order $s>1$ or analytic ($s=1$), the above decay is in fact sub-exponential or exponential, respectively. We also show by a counterexample that ultra-analytic regularity ($s<1$) does not give super-exponential decay. This is in sharp contrast to the more favorable case of pseudodifferential operators, or even (generalized) metaplectic operators, which are treated as well.
\end{abstract}

\title{Exponentially sparse representations of Fourier integral operators}

\author{Elena Cordero,  Fabio Nicola  and Luigi Rodino}
\address{Dipartimento di Matematica,
Universit\`a di Torino, via Carlo Alberto 10, 10123 Torino, Italy}
\address{Dipartimento di Scienze Matematiche,
Politecnico di Torino, corso Duca degli Abruzzi 24, 10129 Torino,
Italy}
\address{Dipartimento di Matematica,
Universit\`a di Torino, via Carlo Alberto 10, 10123 Torino, Italy}

\email{elena.cordero@unito.it}
\email{fabio.nicola@polito.it}
\email{luigi.rodino@unito.it}

\maketitle
\section{Introduction}
We consider Fourier integral operators (FIOs) in the reduced form 
\begin{equation}\label{uno}
Tf(x)=\int_{\rd} e^{2\pi i \Phi(x,\eta)}\sigma(x,\eta)\widehat{f}(\eta)\,d\eta
\end{equation}
of the type of those in \cite{AS78}, namely the amplitude $\sigma(z)$, $z=(x,\eta)$, is in $S^0_{0,0}$, i.e. $\partial^\alpha_z\sigma(z)$ is bounded for every $\alpha$ and the real-valued phase function $\Phi(z)$, satisfying the standard nondegeneracy condition, belongs to $S^{(2)}_{0,0}$, i.e. $\partial^\alpha_z \Phi(z)$ is bounded for $|\alpha|\geq 2$. Such FIOs represent the propagators at a fixed time $t>0$, for the Schr\"odinger equations
\begin{equation}\label{due}
D_t u+a^w(t,x,D_x) u=0,\quad u|_{t=0}=f(x),
\end{equation} 
with real-valued Hamiltonian $a(t,x,\xi)$ belonging to $S^{(2)}_{0,0}$ uniformly in $t$, see for example Tataru \cite{7mail} and Bony \cite{bony1,bony3}.  \par
In \cite{fio3,7mail} it was proved that the Gabor matrix representation of $T$ is concentrated along the graph of the canonical transformation $\chi$ determined by $\Phi$, and provides optimal sparsity. The Gabor representation was then used to discuss the boundedness properties of $T$, cf.\ \cite{fio5,fio1}, and define Wiener algebras of global FIOs containing operators of type \eqref{uno}, see \cite{jmpa}. \par
In the present paper, cf.\ Section 3, we shall present a stronger sparsity result, with exponential decay, for the case of analytic-Gevrey functions, namely when we have in \eqref{uno} for some $s\geq1$, 
\begin{equation}\label{tre}
|\partial^\alpha\Phi(z)|\lesssim C^{|\alpha|}(\alpha!)^s,\quad \alpha\in\mathbb{N}^{2d},\ |\alpha|\geq 2,\ z\in\rdd
\end{equation}
and similarly for the amplitude: 
\begin{equation}\label{quattro}
|\partial^\alpha \sigma(z)|\lesssim C^{|\alpha|}(\alpha!)^s,\quad \alpha\in\mathbb{N}^{2d},\ z\in\rdd.
\end{equation}
As a side result, we shall deduce boundedness of $T$ on $S^s_s(\rd)$, Gelfand-Shilov spaces (basic definitions and properties for these spaces are recalled in the preliminary Section 2). \par
We shall not give explicit applications to the general Schr\"odinger equation \eqref{due} in the present paper. As a matter of fact, a precise version of \cite{AS78} in analytic-Gevrey case, i.e. when the estimates of the type \eqref{tre} are satisfied by the Hamiltonian $a(t,x,\xi)$, is missing in the literature as far as we know. Note however that there is a number of papers where the Schr\"odinger propagators are treated in the analytic framework under decay assumptions for $a$, $\Phi$, $\sigma$, see for example \cite{K-T,M-N-S2,M-N-S2bis,M-R-Z1ter,R-Z1,R-Z1bis}; cf.\ also \cite{C-N1,C-N2} and \cite[Chapter 6]{NR} for standing wave solutions. \par
Let us also mention the reach literature concerning the different case of the H\"orman\-der's FIOs \cite{hormander}, i.e. positive homogeneity of degree $1$ with respect to $\eta$ for $\Phi(x,\eta)$ and corresponding decay estimates in \eqref{tre}, \eqref{quattro}, mainly addressed to the study of the hyperbolic equations. For such FIOs in the analytic-Gevrey category see the bibliography of \cite{rodino-book} concerning the intensive production of the years '80-'90. The researches in this area are indeed extremely active also nowadays, with applications to weakly hyperbolic problems in Gevrey classes. \par

In the above mentioned literature, the analytic regularity $s=1$ is regarded as optimal result. Instead, when dealing with the Gelfand-Shilov classes $S^s_s(\rd)$, it is natural to question whether we can go beyond the barrier $s=1$, getting super-exponential sparsity and ultra-analytic regularity, i.e. boundedness in Gelfand-Shilov spaces for $1/2\leq s <1$. As we shall clarify in Section 4, this is possible if and only if the phase function $\Phi(x,\eta)$ is quadratic in $x,\eta$. Such propagators are obtained from \eqref{due} when $a(t,x,\xi)$ has quadratic principal part in $(x,\xi)$. The corresponding operators $T$ in \eqref{uno}, with amplitude satisfying \eqref{quattro} for $1/2\leq s<1$, are studied in Section 5. \par
In the second part of this introduction we want to give a short presentation to Gabor frames, addressing to non-expert readers. Generally speaking: paradigm of the applications of harmonic analysis to the study of operators and function spaces is the decomposition/reconstruction into ``wave packets'': Fourier series, wavelets, paraproducts, etc., see the survey work \cite{1mail} and also \cite{2mail} for applications to dispersive equations and the restriction theorem. We may say that every class of symbols, i.e. every class of partial differential equations, requires a corresponding partition of the phase space into wave packets, cf.\ \cite{5mail}. Consequently, we may represent the propagator as an infinite matrix, and the chosen partition works effectively for the problem under investigation if the matrix is sparse and well-organized. This means that the propagator re-arranges the wave packets with a controlled number of overlapping of supports, granting continuity on function spaces.\par
For wave equations and H\"ormander's FIOs let us mention the pioneering work of Cord\'oba and Fefferman \cite{B18}, the second dyadic decomposition of Seeger, Sogge and Stein \cite{6mail} and the phase space transform of Tataru and Geba \cite{4mail}; see also Tataru \cite{3mail} for applications to wave equations with non-smooth coefficients.\par
Numerically stable treatments of H\"ormander's FIOs were carried on by Cand\`es, Demanet \cite{candes,cddy} and Guo, Labate \cite{guo-labate}, wave packets being represented by curvelets and shearles.
\par
Gabor frames, used initially for problems in Signal Theory and Time-frequency Analysis, cf.\ \cite{ibero13,gabor,grochenig}, turn out to be the correct setting for Schr\"odinger propagators, at least when in the Hamiltonian the space variables $x$ and their duals $\xi$ play symmetric role, as we have in the $S^0_{0,0}$ class. This means that the microlocal propagation of singularities is identified by the canonical transformation, modulo errors which we may estimate parithetically in the $x$ and $\xi$ variables, see 
 \cite{bony1,bony3,fio5,jmpa,fio1,fio3,7mail}, mentioned before.

To be definite, let us recall some basic definition. 
Let $\Lambda=A\zdd$  with $A\in GL(2d,\R)$ be a lattice
of the time-frequency plane. Consider the time-frequency-shifts
\begin{equation}\label{intro1}
g_{\lambda}=g_{\lambda_1,\lambda_2}=e^{2\pi i \lambda_2 x}g(x-\lambda_1),\quad \lambda=(\lambda_1,\lambda_2)\in\Lambda.
\end{equation}
 The set  of
time-frequency shifts $\G(g,\Lambda)=\{g_\lambda:\,
\lambda\in\Lambda\}$ for a  non-zero $g\in L^2(\rd)$ is called a
Gabor system. The set $\G(g,\Lambda)$   is
a Gabor frame, if there exist
constants $A,B>0$ such that
\begin{equation}\label{gaborframe}
A\|f\|_2^2\leq\sum_{\lambda\in\Lambda}|\langle f,g_\lambda\rangle|^2\leq B\|f\|^2_2,\qquad \forall f\in L^2(\rd).
\end{equation}
 If \eqref{gaborframe} is satisfied, then there exists a dual window $\gamma\in L^2(\rd)$, such that $\cG(\gamma,\Lambda)$ is a frame, and every $f\in L^2(\rd)$ possesses the frame expansions
 \[
 f=\sum_{\lambda\in\Lambda}\langle f,g_\lambda\rangle\gamma_\lambda=\sum_{\lambda\in\Lambda}\langle f,\gamma_\lambda\rangle g_\lambda
 \]
 with unconditional convergence in $L^2(\rd)$.\par
 The Gabor decomposition of an operator $T$ is then as follows:
$$Tf(x)=\sum_{\mu\in\Lambda}\sum_{\lambda\in\Lambda}\underbrace{\la T g_\lambda,g_\mu\ra}_{ M_{\mu \,\lambda}} c_\lambda \gamma_\mu,
\quad \mbox{with}\quad c_\lambda =\la f,\gamma_\lambda\ra.$$

So the action of the operator $T$  above can be read on the coefficient space as
 \[
 \{c_{\lambda}\}_{\lambda\in\Lambda}\mapsto\big\{\sum_{\lambda\in\Lambda}\langle T g_\lambda,g_\mu\rangle c_\lambda\big\}_{\mu\in\Lambda},
 \]
 i.e. it is represented as the infinite matrix $\{M_{\mu \,\lambda}\}_{\mu,\lambda\in\Lambda}=\{\langle T g_\lambda,g_\mu\rangle\}_{\mu,\lambda\in\Lambda}$, which we call the Gabor matrix of $T$. \par
 We can now describe the main results of the paper. Let $T$ be defined as in \eqref{uno}, with non-degenerate phase function $\Phi$ satisfying \eqref{tre} and amplitude $\sigma$ satisfying \eqref{quattro}. Let $\chi:\rdd\to\rdd$ be the canonical transformation associated to $\Phi$. In the generic case $s\geq1$, fix a window $g\in S^{s/2}_{s/2}(\rd)$. Then for some $\epsilon>0$
 \begin{equation}\label{ast}
 |M_{\mu\lambda}|\lesssim \exp\big(-\epsilon|\mu-\chi(\lambda)|^{1/s} \big);
 \end{equation}
 see Theorem \ref{CGelfandPseudo} below. Besides, if $\Phi$ is quadratic then \eqref{ast} keeps valid for $s\geq 1/2$, for any choice of the window $g $ in $S^s_s(\rd)$, see Theorem \ref{CGelfandPseudometap}. The Gaussian, belonging to $S^{1/2}_{1/2}(\rd)$ would work as window in any case. Sparsity and boundedness follow easily, see Propositions \ref{propo} and \ref{gsooo}, whereas we refer to \cite{fio5} for applications to the problem of propagation of analytic singularities for Schr\"odinger equations.\par
  A class of counterexamples to the validity of \eqref{ast} when $s<1$ and $\Phi$ is not a quadratic polynomial is given in Proposition \ref{nometa}.\par
 \section{Preliminaries}
\subsection{Notations}
The Schwartz class is denoted by
$\sch(\Ren)$, the space of tempered
distributions by  $\sch'(\Ren)$.   We
use the brackets  $\la f,g\ra$ to
denote the extension to $\sch '
(\Ren)\times\sch (\Ren)$ of the inner
product $\la f,g\ra=\int f(t){\overline
{g(t)}}dt$ on $L^2(\Ren)$.

 The Fourier
transform is normalized to be ${\hat
  {f}}(\o)=\Fur f(\o)=\int
f(t)e^{-2\pi i t\o}dt$.

Translation and modulation operators, $T$ and $M$ are defined by
$$
T_x f(\cdot) = f(\cdot - x) \;\;\; \mbox{ and } \;\;\;
 M_x f(\cdot) = e^{2\pi i x \cdot} f(\cdot), \;\;\; x \in {\bR}^d.
$$
The following relations hold
\begin{equation}\label{base}
M_y T_x  = e^{2\pi i x  y } T_x M_y, \;\;
 (T_x f)\hat{} = M_{-x} \hat f, \;\;
  (M_x f)\hat{} = T_{x} \hat f, \;\;\;
  x,y \in {\bR}^d,  f,g \in L^2 ({\bR}^d).
\end{equation}

Throughout the paper, we
shall use the notation
$A\lesssim B$ to express the inequality
$A\leq c B$ for a suitable
constant $c>0$, and  $A
\asymp B$  for the equivalence  $c^{-1}B\leq
A\leq c B$.

The letter $ C $ denotes a positive constant, not necessarily the same at
every appearance.\par

\subsection{Gelfand-Shilov Spaces}

Specially in Applied Mathematics, it is of great interest to quantify the decay of functions at infinity, and the Schwartz class $\cS(\rd)$ reveals to be insufficient for this. The so-called Gelfand-Shilov type spaces, introduced in  \cite{GS} turn out to be very useful. Let us recall their definition and main properties; see e.g. \cite{GS,NR} for details.
\begin{definition} Let there be given $ s, r \geq0$.
A function $f\in\cS(\rd)$ is in the Gelfand-Shilov type space $ S^{s} _{r} (\rd) $ if there exist  constants $A,B>0$ such that
\begin{equation}\label{GFdef}
|x^\a\partial^\beta f(x)| \lesssim A^{|\a|}B^{|\beta|}(\a!)^r(\beta!)^s,\quad \a,\beta\in\bN^d.
\end{equation}
\end{definition}
The space $S^{s} _{r}(\rd) $
is nontrivial if and only if $ r + s > 1,$ or $ r + s = 1$ and $r, s > 0$.  So the smallest nontrivial space with $r=s$ is provided by $S^{1/2}_{1/2}(\rd)$. Every function of the type $P(x)e^{-a|x|^2}$, with $a>0$ and $P(x)$ polynomial on $\rd$, is in the class $S^{1/2}_{1/2}(\rd)$.  We observe the trivial inclusions $S^{s_1} _{r_1}(\rd)\subset S^{s_2} _{r_2}(\rd)$ for $s_1\leq s_2$ and $r_1\leq r_2$.
Moreover, if $f\in S^{s} _{r}(\rd)$, also $x^{\delta}\partial^\gamma f$ belongs to the same space for every fixed $\delta,\gamma$. \par The action of the Fourier transform on $S^{s} _{r}(\rd) $ interchanges the indices $s$ and $r$, as explained in the following theorem.

\begin{theorem}\label{ft} For $f\in \cS(\rd)$ we have $f\in S^{s} _{r}(\rd) $ if and only if $\hat{f}\in S^{r}_{s}(\rd).$
\end{theorem}

Therefore for $s=r$ the spaces $S^{s} _{s}(\rd)$ are invariant under the action of the Fourier transform.

We shall also need the following analyticity property of functions in $S^s_r(\rd)$, when $s<1$.
\begin{theorem}[\text{\cite[Proposition 6.1.8]{NR}}]\label{teo0}
Assume $f\in{S}^s_r(\rd)$, $0<s<1$, $r>0$. Then $f$ extends to an entire analytic function $f(x+iy)$ in $\mathbb{C}^d$, with
\begin{equation}\label{eqn:A.21}
\abs{f(x+iy)}\lesssim e^{-\epsilon \abs{x}^{\frac{1}{r}}+\delta\abs{y}^{\frac{1}{1-s}}},\quad x\in\rd,\ y\in\rd,
\end{equation}
where $\epsilon$ and $\delta$ are suitable positive constants.
\end{theorem}

Let us underline the following property, which exhibits  two equivalent ways of expressing the decay of a continuous function $f$ on $\rd$. This follows immediately from \cite[Proposition 6.1.5]{NR}, see also \cite[Proposition 2.4]{fio4}, where the mutual dependence between the constants $\eps$ and $C$ below was shown.
\begin{proposition}[\text{\cite[Proposition 6.1.5]{NR}, \cite[Proposition 2.4]{fio4}}]\label{equi} Consider $r>0$ and let $h$ be a continuous function on $\rd$. Then the following conditions are equivalent:\\
(i) There exists a constant $\eps>0$ such that
\begin{equation}\label{eqn:A.7}
|h(x)|\lesssim e^{-\eps |x|^{\frac1r}},\quad x\in\rd.
\end{equation}
\noindent
(ii)  There exists a constant $C>0$ such that
\begin{equation}\label{powerdecay}
|x^\a h(x)|\lesssim C^{|\a|}(\a!)^{r},\quad x\in\rd, \,\a\in\bN^d.
\end{equation}
Indeed,
assuming \eqref{eqn:A.7}, then \eqref{powerdecay} is satisfied with $C=\displaystyle{\left(\frac { r d}{\eps}\right)^r}$. Viceversa, \eqref{powerdecay}
implies \eqref{eqn:A.7} for any $\epsilon< r(d C)^{-\frac{1}{r}}$. Also, the constant implicit in the notation $\lesssim$ in \eqref{eqn:A.7} depends only on the corresponding one in \eqref{powerdecay} and viceversa.
\end{proposition}

\subsection{Time-frequency characterization of Gelfand-Shilov spaces}  Consider a distribution $f\in\cS '(\rd)$
and a Schwartz function $g\in\cS(\rd)\setminus\{0\}$ (the so-called
{\it window}).
The short-time Fourier transform (STFT) of $f$ with respect to $g$ is
\[
V_gf (x,\o) = \langle f, g_{x,\o}\rangle=\int_{\rd}e^{2\pi i t\o} \overline{g(t-x)}f(t)\,dt\,\qquad (x,\o)\in\rdd.
\]
 The  \stft\ is well-defined whenever the bracket $\langle \cdot , \cdot \rangle$ makes sense for
dual pairs of function or (ultra-)distribution spaces, in particular for $f\in
\cS ' (\rd )$ and $g\in \cS (\rd )$, $f,g\in\lrd$, or $f\in
(S^s_r) ' (\rd )$ and $g\in S^s_r (\rd )$ (see \cite{grochenig} for the full
details). \par
The following inversion formula holds for
the STFT (\cite[Proposition 11.3.2]{grochenig}): assume $g\in \cS(\rd)\setminus\{0\}$,
 $f\in L^2(\rd)$, then
\begin{equation}\label{invformula}
f=\frac1{\|g\|_2^2}\int_{\R^{2d}} V_g f(x,\o)M_\o T_x g\, dx\,d\o.
\end{equation}

\par Finally, we have the following characterization of Gelfand-Shilov functions; cf.\ \cite{elena07, medit,GZ,T2}.
\begin{theorem}\label{teo1}
 If $s\geq1/2$,
 \begin{equation}\label{zimmermann1}
 f,g\in S^s_s(\rd)\,\Rightarrow V_g f\in S^s_s(\rdd);
 \end{equation}
  if $g\in S^s_s(\rd)$, then
 \begin{equation}\label{zimmermann2}
   f\in S^s_s(\rd)\Longleftrightarrow |V_g(f)(z)|\lesssim e^{-\epsilon |z|^{1/s}},\ z\in\rdd,\ \mbox{for some} \,\,\epsilon>0.
 \end{equation}

 \end{theorem}
\section{Exponential sparsity of the Gabor matrix representation}\label{section4}
The
Fourier integral operator $T$ with symbol (or amplitude) $\sigma$  and phase
$\Phi$ on $\rdd$ is formally defined in \eqref{uno}.
The phase
function $\Phi(x,\o)$ is smooth on $\rdd$, and fulfills the estimates
\begin{equation}\label{A}
|\partial^\a \Phi(z)|\lesssim
C^{|\alpha|}(\alpha!)^s,\quad \alpha\in \mathbb{N}^{2d},\ |\a|\geq 2, \
z\in\rdd,
\end{equation}
for some $C>0$, $s\geq 1$, as well as the nondegeneracy condition
\begin{equation}\label{B}
    |\det\,\partial^2_{x,\eta} \Phi(x,\o)|\geq \delta>0,\quad (x,\o)\in\rdd.
\end{equation}
The symbol $\sigma$ on $\rdd$
satisfies
\begin{equation}\label{C}
    |\partial^{\a} \sigma(z)|\lesssim M(z) C^{|\alpha|}(\alpha!)^s,\quad\ \a\in\mathbb{N}^{2d},\
     z\in\rdd,
\end{equation}
for the same $s$ as in \eqref{A} and some $C>0$, and some continuous weight $M>0$ in $\rdd$. We assume here that $M$ is temperate, in the sense that
\begin{equation}\label{D}
M(z+w)\lesssim \langle z\rangle^N M(w),\quad z,w\in\rdd,
\end{equation}
for some $N>0$. \par
 We also denote by $\chi:\rdd\to\rdd$ the canonical transformation defined by $\Phi$, i.e.
\begin{equation}\label{cantra}
(x,\xi)=\chi(y,\eta)  \Longleftrightarrow \left\{
                 \begin{array}{l}
                 y=\nabla_{\eta}\Phi(x,\eta)
                 \\
                \xi=\nabla_{x}\Phi(x,\eta). \rule{0mm}{0.55cm}
                 \end{array}
                 \right.
\end{equation}
The canonical transformation
$\chi $ enjoys the following properties:

\par\medskip
\noindent {\it  (i)} $\chi:\rdd\to\rdd$ is smooth, invertible,  and
preserves the symplectic form in $\rdd$, i.e., $dx\wedge d\xi= d
y\wedge d\eta$; $\chi $ is a \emph{symplectomorphism}.
 \\
{\it  (ii)} For $z=(y,\eta)$,
\begin{equation}\label{chistima}
|\partial_z^\a \chi(z)|\lesssim C^{|\alpha|}(\alpha!)^s,\quad |\a|\geq 1;\end{equation}
{\it (iii)} There exists $\delta>0$ such that, for
$(x,\xi)=\chi(y,\eta)$,
\begin{equation}\label{detcond2}
   |\det\,\frac{\partial x}{\partial y}(y,\eta)|\geq \delta.
\end{equation}
We need a preliminary lemma. 
\begin{lemma}\label{lemmanuovo}
Let $s\geq1$ and $\varphi(z)$ a real smooth function in $\rd$ satisfying the estimates
\[
|\partial^\alpha \varphi(z)|\leq C^{|\alpha|+1}(\alpha!)^s\langle z\rangle^2, \quad \alpha\in\mathbb{N}^d,\ z\in\rd,
\]
for some constant $C>0$. Then for the same constant $C$ it turns out
\[
|\partial^\alpha e^{i\varphi(z)}|\leq (d^{s-1}2^{d+1}C^2)^{|\alpha|}\sum_{j=1}^{|\alpha|}\Big( \frac{\alpha!}{j!}\Big)^s\langle z\rangle^{2j}, \quad |\alpha|\geq 1,\ z\in\rd.
\]
\end{lemma}
\begin{proof}
By the Fa\`a di Bruno formula (see, e.g., \cite[page 16]{GP}) and the hypothesis we have, for $|\alpha|\geq 1$,
\begin{align*}
|\partial^\alpha e^{i\varphi(z)}|&\leq \sum_{j=1}^{|\alpha|} \frac{1}{j!}\sum_{\gamma_1+\ldots+\gamma_j=\alpha\atop |\gamma_k|\geq1}\frac{\alpha!}{\gamma_1!\ldots\gamma_j!}|\partial^{\gamma_1}\varphi(z)|\ldots |\partial^{\gamma_j} \varphi(z)| \\
&\leq \sum_{j=1}^{|\alpha|} \frac{1}{j!}\sum_{\gamma_1+\ldots+\gamma_j=\alpha\atop |\gamma_k|\geq1}\frac{\alpha!}{\gamma_1!\ldots\gamma_j!}C^{|\gamma_1|+\ldots |\gamma_j|+j} (\gamma_1!\ldots \gamma_j!)^s\langle z\rangle^{2j}\\
&=\sum_{j=1}^{|\alpha|} C^{|\alpha|+j}\frac{\alpha!}{j!}\langle z\rangle^{2j}\sum_{\gamma_1+\ldots+\gamma_j=\alpha\atop |\gamma_k|\geq1}(\gamma_1!\ldots\gamma_j!)^{s-1}.
\end{align*}
Now will verify that
\begin{equation}\label{disug}
\gamma_1+\ldots+\gamma_j=\alpha,\quad |\gamma_k|\geq 1\quad\Longrightarrow \quad\gamma_1!\ldots\gamma_j!\leq \frac{|\alpha|!}{j!}.
\end{equation}
This then gives the desired conclusion, taking into account that $s\geq1$,  $|\alpha|!\leq d^{|\alpha|}\alpha!$
and
$$\sum_{\gamma_1+\ldots+\gamma_j=\alpha\atop |\gamma_k|\geq1} 1\leq \prod_{k=1}^d{\alpha_k+j-1 \choose j-1} \leq 2^{|\a|+d(j-1)}\leq 2^{(d+1)|\a|}.$$
 \par
It remains to prove \eqref{disug}. We argue by induction on $j$. If $j=1$ it is obviously true. Let therefore $j\geq 2$ and assume that \eqref{disug} holds for $j-1$ factors. Then
\begin{align*}
\gamma_1!\ldots\gamma_{j-1}!\gamma_j!&\leq \frac{(|\alpha|-|\gamma_j|)!}{(j-1)!}\gamma_j!\leq \frac{(|\alpha|-|\gamma_j|)!}{(j-1)!}|\gamma_j|!\\
&=\frac{|\alpha|!}{j!}\cdot\frac{j}{|\alpha|}\cdot\frac{(|\alpha|-|\gamma_j|)!|\gamma_j|!}{(|\alpha|-1)!}.
\end{align*}
Since $j\leq|\alpha|$ the derided estimate in \eqref{disug} therefore follows if we prove that the last fraction is $\leq1$. But this is clear because
\[
\frac{(|\alpha|-|\gamma_j|)!|\gamma_j|!}{(|\alpha|-1)!}=\frac{|\gamma_j|}{|\alpha|-1}\cdot\frac{|\gamma_j|-1}{|\alpha|-2}\cdot\frac{|\gamma_j|-2}{|\alpha|-3}\ldots\frac{2}{|\alpha|-|\gamma_j|+1}
\]
and in this product each fraction is $\leq 1$: indeed, $j\geq 2$ and $|\gamma_1|\geq 1$ imply $|\gamma_j|\leq |\alpha|-1$ and therefore $|\gamma_j|-k\leq |\alpha|-1-k$, for $0\leq k\leq |\gamma_j|-2$.
\end{proof}
\begin{remark}\rm Let us observe that the Fa\`a di Bruno formula, combined with the formula \eqref{disug}, gives a cheap proof that Gevrey classes are stable by functional composition, with precise estimates for the constants; we omit the details. 
\end{remark}

\begin{theorem}\label{CGelfandPseudo}  Let $s\geq1$, and suppose the phase $\Phi$ and symbol $\sigma$ satisfy \eqref{A}--\eqref{D} above. Assume $g\in S^{s/2}_{s/2}(\rd)$  Then there exists $\eps>0$ such that
\begin{equation}\label{unobis2s} |\langle T
g_{u},g_{v}\rangle|\lesssim M(v_1,u_2)\exp\big(-\eps|v-\chi(u)|^{1/s}\big),\quad
u=(u_1,u_2),\ v=(v_1,v_2)\in\rdd.
\end{equation}
\end{theorem}
\begin{proof}
A direct computation based on \eqref{base} (see e.g.\ the proof of \cite[Theorem 3.1]{fio1}) shows that
\begin{align*}
  \la T g_{u},&  g_{v}\ra \\
    =&\intrd\intrd e^{2\pi i [\Phi(x+v_1,\o+u_2)-(v_2,u_1)\cdot (x+v_1,\o)]}\sigma(x+v_1,\o+u_2) \bar{g}(x)\hat{g}(\o)\,dx
   d\o
\end{align*}
By performing a Taylor expansion of $\Phi$ around
$(v_1,u_2)$ we obtain
\begin{equation}\label{E}
|\langle T
g_{u},g_{v}\rangle|=\Big|\int_{\rdd} e^{2\pi i(\nabla_z \Phi(v_1,u_2)-(v_2,u_1))z} e^{2\pi i\Phi_{2,(v_1,u_2)}(z)} \sigma(z+(v_1,u_2))G(z)\,dz\Big|
\end{equation}
where $G(z)=G(x,\o)=\overline{g}(x)\otimes \widehat{g}(\o)$, and
 \begin{equation}
  \label{eq:c11}
\Phi_{2,(v_1,u_2)}(z)=2\sum_{|\a|=2}\int_0^1(1-t)\partial^\a
\Phi((v_1,u_2)+tz)\,dt\frac{z^\a}{\a!},\quad z=(x,\o),
\end{equation}
is the second order remainder in the Taylor formula for $\Phi$ at $(v_1,u_2)$. Observe that \eqref{A} implies the estimates
\begin{equation}\label{F}
|\partial^{\alpha} \Phi_{2,(v_1,u_2)}(z)|\lesssim C^{|\alpha|}(\alpha!)^s\langle z\rangle^2,\quad \alpha\in\mathbb{N}^{2d},\ z\in\rdd,
\end{equation}
uniformly with respect to $v_1,u_2\in\rd$.\par
Now, it is proved in \cite[Lemma 3.1]{fio1} that
\[
|\nabla_z \Phi(v_1,u_2)-(v_2,u_1)|\gtrsim |v-\chi(u)|.
\]
Hence, it is sufficient to prove that
\begin{equation*}\label{30-1} |\langle T
g_{u},g_{v}\rangle|\lesssim M(v_1,u_2)\exp\big(-\eps|\nabla_z \Phi(v_1,u_2)-(v_2,u_1)|^{1/s}\big),\quad
u,v\in\rdd.
\end{equation*}
Using the formula \eqref{E} for the left-hand side, we are reduced to proving that the function
\[
h_{v_1,u_2}(\omega):=\int_{\rdd} e^{2\pi i \omega z} e^{2\pi i \Phi_{2,(v_1,u_2)}(z)} \sigma (z+(v_1,u_2))G(z)\, dz,\quad \omega\in\rdd
\]
satisfies the estimates
\[
|h_{v_1,u_2}(\omega)|\lesssim M(v_1,u_2)\exp\big(-\epsilon|\omega|^{1/s}\big)
\]
or equivalently, by Proposition \ref{equi},
\begin{equation}\label{Fbis}
|\omega^\alpha h_{v_1,u_2}(\omega)|\lesssim M(v_1,u_2) C^{|\alpha|}(\alpha!)^s,\quad \alpha\in\mathbb{N}^{2d},\ \omega\in\rdd.
\end{equation}
Now, repeated integrations by parts and Leibniz formula give
\begin{multline}\label{G}
|\omega^\alpha h_{v_1,u_2}(\omega)|\leq (2\pi)^{-|\alpha|}\Big|\int_{\rdd} \langle z\rangle^{-2d-1} e^{2\pi i \omega z}\sum_{\beta_1+\beta_2+\beta_3=\alpha}\frac{\alpha!}{\beta_1!\beta_2!\beta_3!} \\  \times\langle z\rangle^{2d+1}\partial^{\beta_1}e^{2\pi i \Phi_{2,(v_1,u_2)}(z)} \partial^{\beta_2}\sigma (z+(v_1,u_2))\partial^{\beta_3}G(z)\, dz\Big|.
\end{multline}
Let us estimate the three derivatives above. By \eqref{F} and Lemma \ref{lemmanuovo} we have
\begin{equation}\label{stimaesf}
|\partial^{\beta_1}e^{2\pi i \Phi_{2,(v_1,u_2)}(z)}|\lesssim C^{|\beta_1|} \sum_{j=1}^{|\beta_1|}\Big(\frac{\beta_1!}{j!}\Big)^s \langle z\rangle^{2j},\quad |\beta_1|\geq1.
\end{equation}
Using \eqref{C} and \eqref{D}, the derivatives of the symbol can be controlled by
\[
|\partial^{\beta_2}\sigma (z+(v_1,u_2))|\leq M(v_1,u_2)C^{|\beta_2|}(\beta_2!)^s\langle z\rangle^{N}.
 \]
 Hence, for $|\beta_1|\geq1$,
\begin{multline}\label{H}
|\langle z\rangle^{2d+1}\partial^{\beta_1}e^{2\pi i \Phi_{2,(v_1,u_2)}(z)} \partial^{\beta_2}\sigma (z+(v_1,u_2))\partial^{\beta_3}G(z)|\\ \lesssim M(v_1,u_2)C^{|\beta_1|+|\beta_2|}(\beta_2!)^s \langle z\rangle^{N+2d+1} \sum_{j=1}^{|\beta_1|}\Big(\frac{\beta_1!}{j!}\Big)^s \langle z\rangle^{2j}|\partial^{\beta_3}G(z)|.
\end{multline}
Now, by Theorem \ref{ft},  we have $\widehat{g}\in S^{s/2}_{s/2}(\rd)$, so that $G=\overline{g}\otimes\widehat{g}\in S^{s/2}_{s/2}(\rdd)$. This gives
\begin{align*}
\langle z\rangle^{N+2d+1+2j}|\partial^{\beta_3}G(z)|&\lesssim C^{N+2d+1+2j+|\beta_3|}((N+2d+1+2j)!)^{s/2}(\beta_3!)^{s/2}\\
&\lesssim {C_1}^{2j+|\beta_3|}(j!)^s (\beta_3!)^{s/2}
\end{align*}
where we used the formula $(m+n)!\leq 2^{m+n}m!n!$ and Stirling formula. Hence
\begin{align}\label{inpiu}
\sum_{j=1}^{|\beta_1|}\Big(\frac{\beta_1!}{j!}\Big)^s \langle z\rangle^{N+2d+1+2j}|\partial^{\beta_3}G(z)|&\lesssim
\sum_{j=1}^{|\beta_1|}\Big(\frac{\beta_1!}{j!}\Big)^s {C_1}^{2j+|\beta_3|}(j!)^s (\beta_3!)^{s/2}\\&\lesssim {C_2}^{|\beta_1|+|\beta_3|}(\beta_1!)^s (\beta_3!)^{s/2}\nonumber
\end{align}
for a suitable $C_2>1$,  where we used $\sum_{j=1}^{|\beta_1|}1=|\beta_1|-1\leq 2^{|\beta_1|}$.\par
For $|\beta_1|\geq 1$, the estimate \eqref{H} can then be controlled by
\begin{multline}\label{I}
|\langle z\rangle^{2d+1}\partial^{\beta_1}e^{2\pi i \Phi_{2,(v_1,u_2)}(z)} \partial^{\beta_2}\sigma (z+(v_1,u_2))\partial^{\beta_3}G(z)|\\ \lesssim M(v_1,u_2)C^{|\beta_1|+|\beta_2|+|\beta_3|}(\beta_1!\beta_2!\beta_3!)^s\leq M(v_1,u_2)C^{|\alpha|}(\alpha!)^s.
\end{multline}
for a new constant $C>1$. An easier argument shows that the same estimate holds for $\beta_1=0$ too. \par
Finally, the desired result \eqref{Fbis} is obtained by using  the estimate \eqref{I} in \eqref{G}, together with $\sum_{\beta_1+\beta_2+\beta_3=\alpha}\frac{\alpha!}{\beta_1!\beta_2!\beta_3!}= 3^{|\alpha|}$.
\end{proof}

We now show two immediate byproducts of the above theorem, namely, exponential sparsity of the Gabor matrix representation of $T$ and the continuity on the Gelfand-Shilov spaces.
Let therefore $\G(g,\Lambda)$ be a Gabor frame for $\lrd$, with $g\in S^{s/2}_{s/2}(\rd)$, $s\geq 1$.  Under the assumptions of the previous theorem when $M\equiv1$, we have therefore the estimates
\begin{equation}\label{unobis3}
 |\langle T
g_{u},g_{v}\rangle|\lesssim\exp\big(-\eps|v-\chi(u)|^{1/s}\big),
\end{equation}
valid for $u,v\in\rdd$, in particular for $u,v\in\Lambda$.\par
It is easy to see that this implies the sparsity for the Gabor matrix in the classical -- i.e. superpolynomial -- sense; cf.\ \cite{candes, guo-labate}. Actually, here we obtain a sparsity of exponential-type, as detailed in the following result.
\begin{proposition}\label{propo}
Let $\G(g,\Lambda)$ be a Gabor frame for $\lrd$, with $g\in S^{s/2}_{s/2}(\rd)$, $s\geq 1$. Under the assumptions of Theorem \ref{CGelfandPseudo}, with $M\equiv1$, the Gabor matrix
$\langle T g_{\lambda},g_{\mu}\rangle$ is sparse in the following sense.
Let $a$ be any column
or raw of the matrix, and let
$|a|_n$ be the $n$-largest
entry of the sequence $a$.
Then $|a|_n$
satisfies
\[
|a|_n\leq C \displaystyle \exp\big({-\epsilon n^{1/(2ds)}}\big),\quad n\in\bN,
\]
for some constants $C>0,\epsilon>0$.
\end{proposition}
Indeed, this was shown in detail in \cite[Proposition 4.5]{fio4} for any matrix satisfying an estimate of the type \eqref{unobis3}.\par
Another consequence of Theorem \ref{CGelfandPseudo} and the characterization \eqref{zimmermann2} is a continuity result on Gelfand-Shilov spaces.
\begin{proposition}\label{gsooo}
Let $s\geq 1$, and consider a symbol $\sigma\in C^\infty(\rdd)$ and a phase $\Phi$ satisfying the assumptions \eqref{A}, \eqref{B} and \eqref{C} with $M\equiv1$. Then the corresponding Fourier integral operator $T$ in bounded on $S^s_s(\rd)$.
\end{proposition}
\begin{proof}The proof is analogous to the corresponding result for pseudodifferential operators obtained in \cite[Propositions 4.7]{fio4} (but here we restrict to $s\geq1$). In short: from the inversion formula \eqref{invformula}, we get
\[
V_g (T f)(v)=\int_{\R^{2d}}\langle T g_u,
g_v\rangle \, V_g f(u) \, du,
\]
with $g(x)=e^{-\frac{\pi}{2}|x|^2}$, say. 
The estimate \eqref{unobis3} together with the characterization in \eqref{zimmermann2} then give the desired conclusion. 
\end{proof}
\section{A counterexample to super-exponential decay}
In this section we show that there is not a reasonable extension of Theorem \ref{CGelfandPseudo} to the case $s<1$. In other terms, ultra-analytic phases and symbols generally do not give super-exponential decay in \eqref{unobis2s}, even for ultra-analitic windows. \par
Consider, in dimension $d=1$, any real-valued function $\varphi(x)$, $x\in \R$, satisfying the following estimates:
\begin{equation}\label{ipo}
|\varphi^{(\alpha)}(x)|\leq C^{|\alpha|+1}(\alpha!)^{s},\quad \forall \a\geq 2,
\end{equation}
for some $s<1$ (e.g. $\varphi(x)=\cos x$).  Let $T$ be the FIO with phase $\Phi(x,\o)=x\eta+\varphi(x)$ and symbol $\sigma\equiv1$, therefore $Tf(x)=e^{2\pi i \varphi(x)}f(x)$, $\chi(y,\eta)=(y,\eta+\nabla\varphi(y))$.  Observe that the assumptions \eqref{A},\eqref{B} are fulfilled, as well as \eqref{C} with $M\equiv1$. Then, the following holds true.
\begin{proposition}\label{nometa}
For the above operator $T$, suppose the following estimate holds for some $1/2\leq s'<1$, $g\in S^{s'}_{s'}(\R)\setminus\{0\}$, $\epsilon>0$:
\begin{equation}\label{unobis2s2} |\langle T
g_{u},g_{v}\rangle|\lesssim \exp\big(-\eps|v-\chi(u)|^{1/s'}\big),\quad
u,v\in\R^{2},
\end{equation}
 Then $\varphi(x)$ is a polynomial of degree at most $2$.
\end{proposition}
\begin{proof}
The estimate \eqref{unobis2s2}  implies that if $f\in S^{s'}_{s'}(\R)$ then $Tf\in S^{s'}_{s'}(\R)$ (see  the proof of Proposition \ref{gsooo} or \cite[Propositions 4.7]{fio4}). Let now $f(x)=e^{-x^2}\in S^{1/2}_{1/2}(\R)\subseteq S^{s'}_{s'}(\R)$; then $Tf(x)=e^{2\pi i \varphi(x)}e^{-x^2}\in S^{s'}_{s'}(\R)$.
The hypothesis \eqref{ipo} with $s<1$ and Cauchy's estimates imply that $\varphi(x)$ extends to an entire function $\varphi(z)$, $z=x+iy\in\mathbb{C}$. By Theorem \ref{teo0} the function $e^{2\pi i\varphi(z)} e^{-z^2}$ satisfies the growth estimate
\[
|e^{2\pi i\varphi(z)} e^{-z^2}|\leq C e^{-cx^2+C|y|^{\mu}},\quad z=x+iy,
\]
for some constants $C,c>0$, with $\mu=1/(1-s')$. The left-hand side is equal to $e^{-2\pi {\rm Im} \varphi(z)-x^2+y^2}$, hence
\[
-{\rm Im}\, \varphi(z)\leq C(1+x^2+|y|^\mu).
\]
for a new constant $C>0$. A similar estimate holds with $-\varphi$ in place of $\varphi$, because $-\varphi$ satisfies the same assumptions as $\varphi$ and $e^{-2\pi i\varphi(x)}e^{-x^2}\in S^{s'}_{s'}(\R)$ too. Therefore we get
\[
|{\rm Im}\, \varphi(z)|\leq C(1+x^2+|y|^\mu).
\]
So, ${\rm Im}\, \varphi(z)$ has at most an algebraic growth, and the same must hold for the real part ${\rm Re}\, \varphi(z)$, by the Cauchy-Riemann equations. As a consequence, $|\varphi(z)|$ has at most an algebraic growth, therefore $\varphi(z)$ is a polynomial by the Liouville theorem.
Since the second derivative $\varphi''(x)$
is bounded by \eqref{ipo}, $\varphi(x)$ must have degree at most 2. \end{proof}

The above result shows that there is no hope to obtain super-exponential decay except for quadratic phases. Indeed, $T$ is then a metaplectic operator and for those operators we are able to obtain optimal estimates for the corresponding Gabor matrix decay, as explained in the following section.

\section{A class of generalized metaplectic operators}\label{metapoper}
We will study  the  class of
Fourier integral operators  whose canonical transformation  is a {\it linear} transformation $\chi(z) = \cA
z$ for some invertible matrix $\cA \in GL(2d,\R)$.
Since $\chi $ must preserve the symplectic form (assumption $(ii)$), $\cA
$ must be a symplectic matrix, i.e. an element of the symplectic group
$$
\Spnr=\left\{\A\in\Gltwonr:\;^t\!\A J\A=J\right\},
$$
where
$$
J=\begin{pmatrix} 0&-I_d\\I_d&0\end{pmatrix}
\, .
$$
For $z=(x,\xi)$ we shall also write
\[
\pi(z)f=M_{\xi} T_x f.
\]
Given $\cA \in \Spnr $, the
metaplectic operator $\mu (\cA )$ is  defined by the intertwining
relation
\begin{equation}\label{metap}
\pi (\cA z) = c_\cA  \, \mu (\cA ) \pi (z) \mu (\cA )\inv  \quad  \forall
z\in \R ^d \, ,
\end{equation}
where $c_\cA \in \bC , |c_{\cA } | =1$ is a phase factor (for details, see e.g. \cite{folland89}).\par
If $\chi=\A=\begin{pmatrix} A&B\\C&D\end{pmatrix}\in Sp(d,\R)$, then $(x,\xi)=(Ay+B\eta,Cy+D\eta)$ and $\det\,\frac{\partial x}{\partial y}(y,\eta)=\det A$, so that the condition \eqref{detcond2} becomes $\det A\not=0$.\par
Viceversa, to every matrix $\A=\begin{pmatrix} A&B\\C&D\end{pmatrix}\in Sp(d,\R)$ with $\det A\not=0$ corresponds a metaplectic operator $\mu(\A)$  which is a Fourier integral operator of the type \eqref{uno}, as proved in Theorem 4.51 and subsequent \emph{Remark 2} of
\cite{folland89}, recalled below.
\begin{theorem}\label{metafoll} Let
$\mathcal{A}=\begin{pmatrix} A&B\\C&D\end{pmatrix}\in Sp(d,\R)$.
 If  $\det A\not=0$ we have
\begin{equation}\label{f4}
\mu(\A)f(x)=(\det A)^{-1/2}\int e^{2\pi i \Phi(x,\eta)}\hat{f}(\eta)\,d\eta,
\end{equation}
with
\begin{equation}\label{fase}\Phi(x,\eta)=\frac12  x CA^{-1}x+
\eta  A^{-1} x-\frac12\eta  A^{-1}B\eta.
\end{equation}
\end{theorem}

Solving \eqref{cantra} for the phase function in \eqref{fase} we
obtain  $\chi=\A$, as expected.

Observe that  the  phase $\Phi$ in \eqref{fase}  satisfies conditions \eqref{A} and \eqref{B} and the symbol $\sigma\equiv (\det A)^{-1/2}$ in \eqref{f4} fulfills \eqref{C} with exponent $s=0$ (so also for $s=1/2$) and the weight $M\equiv1$.
So these metaplectic operators satisfy the assumptions of Theorem \ref{CGelfandPseudo} for $s\geq 1$ but the best decay result would be
\begin{equation}\label{unobis2s3} |\langle \mu(\A)
g_{u},g_{v}\rangle|\lesssim \exp\big(-\eps|v-\chi(u)|\big),\quad
u,v\in\rdd,
\end{equation}
provided $g\in S^{1/2} _{1/2}(\rd)$. This decay result in not optimal, as  shown by the following motivating example. 
\begin{example} \rm Consider the Cauchy problem for the harmonic oscillator: \begin{equation}\label{pc}
\begin{cases} i \displaystyle\frac{\partial
u}{\partial t} -\frac{1}{4\pi}\Delta u+\pi|x|^2 u=0\\
u(0,x)=u_0(x),
\end{cases}
\end{equation}
 with $(t,x)\in\bR\times\bR^d$, $d\geq1$. For every fixed $t$, the solution:
$$
u(t,x)=(\cos t)^{-d/2}\intrd e^{2\pi i [\frac1{\cos t} x\eta + \frac{\tan t}2 (x^2+\eta^2)]}\hat{f}(\eta)\,d\eta,\quad t\not=\frac\pi 2 + k\pi,\,\,k\in\bZ$$
can be seen as a FIO of type \eqref{uno} with phase $\Phi_t\phas=\displaystyle{\frac1{\cos t} x\eta+
\frac{\tan t}2 (x^2+\eta^2)}$ and symbol $\sigma_t=(\cos t)^{-d/2}$. The associate canonical transformation is
\[\chi_t(y,\eta)=
\left(
\begin{array}{cc}
(\cos t)I & (-\sin t)I\\	
(\sin t)I  & (\cos t)I
\end{array}
\right)
\left(
\begin{array}{c}
y\\	
\eta
\end{array}
\right).
\]
With $g(x)=e^{-\frac{\pi}{2} |x|^2}$, an explicit computation shows the Gaussian decay
$$|\la u(t,\cdot)g_u,g_v\ra|\leq 2^{-\frac d 2} \exp\big(-\frac\pi 2 |v-\chi_t(u)|^2\big),\quad\forall u,v\in\rdd.
$$
\end{example}
More generally, consider the case of a FIO  $T$ with phase $\Phi$ in \eqref{fase} and  symbol $\sigma$ that satisfies \eqref{C}, \eqref{D}, that generalizes the classical metaplectic operator above, having a non-constant symbol. 
\begin{theorem}\label{CGelfandPseudometap}  Let $s\geq1/2$,  consider a FIO T with phase $\Phi$ in \eqref{fase} and  symbol $\sigma$ that satisfies \eqref{C}, \eqref{D}. Assume $g\in S^{s}_{s}(\rd)$.  Then there exists $\eps>0$ such that
\begin{equation*} |\langle T
g_{u},g_{v}\rangle|\lesssim M(v_1,u_2)\exp\big(-\eps|v-\chi(u)|^{1/s}\big),\quad\forall
u,v\in\rdd.
\end{equation*}
\end{theorem}
\begin{proof} The proof uses the same pattern of the one of Theorem \ref{CGelfandPseudo}. 
The matrix of the second order derivatives of the phase $\Phi$ is
$$\left( \partial^\a \Phi \right)_{|\a|=2}=\begin{pmatrix} CA^{-1}&(^t\!A)^{-1}\\A^{-1}&A^{-1}B\end{pmatrix}
$$
and the phase remainder \eqref{eq:c11} becomes $$\Phi_{2,(v_1,u_2)}(z)=\sum_{|\a|=2} c_\a z^\a,\quad c_\a\in\R,\quad \forall (v_1,u_2)\in\rdd.$$
 So $|\partial^\beta (\Phi_{2,(v_1,u_2)}(z))|\leq C \la z\ra^{2-|\beta|}$ for $1\leq|\beta|\leq 2$, whereas $\partial^\beta (\Phi_{2,(v_1,u_2)}(z))=0$ for every $z\in\rdd$ when $|\beta|>2$.
In this case, by the Fa\`a di Bruno formula,  the estimate \eqref{stimaesf} is replaced by
\begin{equation*}
|\partial^{\beta_1}e^{2\pi i \Phi_{2,(v_1,u_2)}(z)}|\leq C^{|\beta_1|} \sum_{j=1}^{|\beta_1|}\frac{\langle z\rangle^{2j-|\beta_1|}}{j!}\sum_{\gamma_1+\ldots+\gamma_j=\beta_1\atop 1\leq|\gamma_k|\leq2}\frac{\beta_1!}{\gamma_1!\ldots\gamma_j!} ,\quad |\beta_1|\geq1,
\end{equation*}
namely 
\[
|\partial^{\beta_1}e^{2\pi i \Phi_{2,(v_1,u_2)}(z)}|\leq C_1^{|\beta_1|}\sum_{j=1}^{|\beta_1|}\frac{\beta_1!}{j!} \langle z\rangle^{2j-|\beta_1|}
\]
for a new constant $C_1>0$. Now we have $\widehat{g}\in S^{s}_{s}(\rd)$, so that $G:=\overline{g}\otimes\widehat{g}\in S^{s}_{s}(\rdd)$, and the analog of formula \eqref{inpiu} is  here
\begin{align*}\label{inpiu2}
\sum_{j=1}^{|\beta_1|} \frac{\beta_1!}{j!} \langle z\rangle^{N+2d+1+2j-|\beta_1|}|\partial^{\beta_3}G(z)|&\lesssim C_2^{|\beta_1|+|\beta_3|}\sum_{j=1}^{|\beta_1|} \frac{\beta_1!}{j!}(2j-|\beta_1|)!^s \beta_3! ^s\\
& \lesssim C_3^{|\beta_1|+|\beta_3|} (\beta_1!\beta_3!)^s
\end{align*}
where we used 
\[
\frac{(2j-|\beta_1|)!^s\beta_1!^{1-s}}{j!}\leq \frac{(2j)!^s\beta_1!^{1-2s}}{j!}\lesssim C_4^{|\beta_1|}\frac{j!^{2s-1}}{|\beta_1|!^{2s-1}}\leq C_4^{|\beta_1|}.
\]
(The last inequality holds because $j\leq |\beta_1|$ and $s\geq 1/2$). 
\end{proof}


\end{document}